\documentclass{amsart}

\usepackage{amsmath, amssymb, amsthm}
\usepackage{enumitem}
\usepackage{color}
\usepackage{url}
\usepackage{stmaryrd}

\newtheorem*{thma}{Theorem~A}

\newtheorem{theorem}{Theorem}[section]
\newtheorem{lemma}[theorem]{Lemma}
\newtheorem{proposition}[theorem]{Proposition}
\newtheorem{corollary}[theorem]{Corollary}
\newtheorem{claim}[theorem]{Claim}

\newtheorem{question}[theorem]{Question}

\theoremstyle{definition}
\newtheorem{definition}[theorem]{Definition}
\newtheorem{remark}[theorem]{Remark}

\newcommand{\dom}{\mathrm{dom}}
\newcommand{\bb}{\mathbb}

\newcommand{\otp}{\mathrm{otp}}

\newcommand{\pred}{\mathrm{pred}}

\newcommand{\mb}{\mathbf}
\newcommand{\is}{\mathrm{IS}}
\newcommand{\mc}{\mathcal}
\renewcommand{\bf}{\mathbf}
\newcommand{\suc}{\mathrm{succ}}
\newcommand{\PG}{\mathsf{PG}}
\newcommand{\DDF}{\mathsf{DDF}}

\title{Polish space partition principles and the Halpern-L\"auchli theorem}
\author{Chris Lambie-Hanson and Andy Zucker}
\address{\newline Institute of Mathematics of the Czech Academy of Sciences \newline
\v{Z}itn\'{a} 25, Prague 1, 115 67, Czech Republic}
\email{lambiehanson@math.cas.cz}
\urladdr{http://math.cas.cz/lambiehanson}
\address{\newline Department of Pure Mathematics\newline University of Waterloo\newline
200 University Ave.\ W.\newline
Waterloo, ON, Canada N2L 3G1}
\email{a3zucker@uwaterloo.ca}

\begin{document}
\maketitle

\begin{abstract}
	The Halpern-L\"auchli theorem, a combinatorial result about trees, admits an elegant proof due to Harrington using ideas from forcing. In an attempt to distill the combinatorial essence of this proof, we isolate various partition principles about products of perfect Polish spaces. These principles yield straightforward proofs of the Halpern-L\"auchli theorem, and the same forcing from Harrington's proof can force their consistency. We also show that these principles are not ZFC theorems by showing that they put lower bounds on the size of the continuum.
\end{abstract}

\section{Introduction}

\let\thefootnote\relax\footnote{2020 Mathematics Subject Classification. Primary: 03E02. Secondary: 05D10.}
\let\thefootnote\relax\footnote{Keywords: Halpern-L\"auchli theorem, Cohen forcing, Polish spaces, partition relations}
\let\thefootnote\relax\footnote{The first author was supported by the institutional grant RVO:67985840. The second author was supported by NSF Grant DMS-2054302.}

The Halpern-L\"auchli theorem, first proven in \cite{halpern_lauchli}, is a partition principle about products of finitely branching trees. While it is entirely combinatorial in nature, it has deep connections to logic, both in its original motivation --- it was a key tool in Halpern and L\'{e}vy's proof \cite{halpern_levy} that, over $\mathsf{ZF}$, the Boolean Prime Ideal theorem does not imply the Axiom of Choice --- and in some methods of proving it. In particular, arguably the most elegant proof of the theorem, due to Harrington, uses some ideas from forcing (see \cite{DobrinenSurvey} or \cite{todorcevic_farah} for a presentation of Harrington's proof). Here, we investigate some of the combinatorial ideas underlying Harrington's proof and, building on work from \cite{hl2d}, we introduce and study a family of statements about arbitrary (i.e.\ not necessarily Borel) finite partitions of a product of finitely many perfect Polish spaces. The simplest of these statements are as in the following definition.

\begin{definition}
	\label{Def:PG}
	Recall that a Polish space is \emph{perfect} if it contains no isolated points. Given $0< d< \omega$ and a sequence $\langle X_0,\ldots,X_{d-1}\rangle$ of perfect Polish spaces, a \emph{somewhere dense grid} is a subset of $\prod_{i< d} X_i$ of the form $\prod_{i<d} Y_i$, where each $Y_i\subseteq X_i$ is somewhere dense (in the ordinary topological sense).
	
	The \emph{Polish grid principle} in dimension $d$, denoted $\mathsf{PG}_d$, is the statement that for any sequence $\langle X_0,\ldots,X_{d-1}\rangle$ of perfect Polish spaces, any $r< \omega$ and any coloring $\gamma\colon \prod_{i< d} X_i\to r$, there is a monochromatic somewhere dense grid. The principle $\mathsf{PG}$ is the statement that $\mathsf{PG}_d$ holds for every $d< \omega$.
	
	More generally, for any cardinal $\kappa$, let $\PG_d(\kappa)$ be the statement that for any sequence $\langle X_0, \ldots, X_{d-1} \rangle$ of perfect Polish spaces and any coloring $\gamma\colon \prod_{i<d}X_i\to \kappa$, there is a monochromatic somewhere dense grid, and let $\PG(\kappa)$ denote the statement that $\PG_d(\kappa)$ holds for every $d < \omega$.
\end{definition}

We shall see in Section~\ref{Sec:PGdef} that $\PG_d$ yields a simple, direct proof of the $d$-dimensional Halpern-L\"{a}uchli theorem; in fact, it will be immediate from the proof that a natural weakening of $\PG_d$, which we will define in Section~\ref{Sec:PGdef} and which we will denote by $\DDF_d$, suffices for this derivation. Unlike the Halpern-L\"{a}uchli theorem, though, at least for $d > 2$ the principles $\PG_d$ and $\DDF_d$ are independent of $\mathsf{ZFC}$. Much of this paper is concerned with the study of this independence; we obtain, for instance, the following complementary results.

\begin{thma}
  Let $2 \leq d < \omega$.
  \begin{enumerate}
    \item If $\DDF_d$ holds, then $2^{\aleph_0} \geq \aleph_{d-1}$.
    \item If $\PG_d(\aleph_0)$ holds, then $2^{\aleph_0} \geq \aleph_d$.
    \item $\PG_d(\aleph_0)$ holds after adding at least $(\beth_{d-1})^+$-many Cohen reals to any model of $\mathsf{ZFC}$.
  \end{enumerate}
\end{thma}

Clause 3 of Theorem A, together with Shoenfield absoluteness, yields a new proof of the Halpern-L\"{a}uchli theorem. This proof can be seen as a recasting of Harrington's proof that seeks to pull apart the forcing machinery and the combinatorial principles underlying the Halpern-L\"{a}uchli theorem. Indeed, our proof of Theorem A(3) is essentially an adaptation of Harrington's proof of the Halpern-L\"{a}uchli theorem.

We feel that Theorem A is also of independent interest for isolating a stratified family of natural combinatorial statements that place increasingly strong requirements on the value of the continuum. In this direction, Theorem A yields a sharp result for the principle $\mathsf{PG}_d(\aleph_0)$: for all $2 \leq d < \omega$, if $2^{\aleph_0} < \aleph_d$, then $\mathsf{PG}_d(\aleph_0)$ fails, whereas it is consistent that $2^{\aleph_0} = \aleph_d$ and $\mathsf{PG}_d(\aleph_0)$ holds (for example, after adding $\aleph_d$-many Cohen reals to a model of $\mathsf{GCH}$). This can be seen as part of a line of investigation into the relationship between questions of \emph{dimensionality} in the context of the real numbers (or uncountable Polish spaces, more generally), and the cardinals $\{\aleph_d : d < \omega\}$. Other research in this vein includes, e.g., Raghavan and Todor\v{c}evi\'c's work on Galvin's problem (\cite{raghavan_todorcevic_2020}, \cite{raghavan_todorcevic_hd}), Komj\'{a}th's work on covering the plane by finite many clouds (\cite{komjath}) and the work of several authors on additive partition relations for the real numbers (\cite{hindman_leader_strauss}, \cite{inf_mono_sumsets}, \cite{zhang}). 

The work of Raghavan and Todor\v{c}evi\'c in particular obtains results with interesting parallels to ours. They consider the problem of finding the \emph{$d$-dimensional Ramsey degrees} of the topological space $\bb{Q}$, i.e.\ given a coloring of $[\bb{R}]^d$ into finitely many colors, find $X\subseteq \bb{R}$ homeomorphic to $\bb{Q}$ such that $[X]^d$ takes a fixed small number of colors (not depending on the starting number of colors). In \cite{raghavan_todorcevic_hd}, they show that $2^{\aleph_0}\leq \aleph_{d-2}$ implies that the $d$-dimensional Ramsey degree of the space $\bb{Q}$ is infinite, and in \cite{raghavan_todorcevic_2020}, they show that in $\mathsf{ZFC}$ plus some large cardinal hypotheses, the $2$-dimensional Ramsey degree of the space $\bb{Q}$ is $2$. 

Note that Theorem A(1) does not yield any nontrivial information from $\DDF_2$ or $\PG_2$. In fact, we shall see in Section~\ref{Sec:PGdef} that $\DDF_2$ is a theorem of $\mathsf{ZFC}$. Moreover, $\PG_2$ is a theorem of $\mathsf{ZFC}$ when restricted to 2-colorings. It remains open whether the full $\PG_2$ is a theorem of $\mathsf{ZFC}$. It would be especially interesting if similar techniques as in  \cite{raghavan_todorcevic_2020} could be used to show that $\mathsf{PG}_2$ is implied by some large cardinal hypotheses.

Finally, the tools developed in our proof of Clauses 1 and 2 of Theorem A yield a new 
proof of a recent result of Bannister, Bergfalk, Moore, and Todor\v{c}evi\'{c} 
from \cite{bbmt} about provable failures of the Partition Hypothesis introduced in that 
paper. The original proof of this fact in \cite{bbmt} makes heavy use of ideas from 
simplicial homology, whereas our proof is purely combinatorial/set theoretic.

The structure of the paper is as follows. In Section~\ref{Sec:prelim}, we give our 
background notational conventions and basic definitions regarding trees and the 
Halpern-L\"{a}uchli theorem. In Section~\ref{Sec:PGdef}, we introduce the partition 
principles that form the main object of study of the paper. We show that they yield 
immediate derivations of relevant instances of the Halpern-L\"{a}uchli theorem. We 
also show that the principle $\DDF_2$ is provable in $\mathsf{ZFC}$, as is the principle 
$\PG_2$ when restricted to 2-colorings. In Section~\ref{Sec:DDFbad}, we prove 
clauses 1 and 2 of Theorem A. We also give a direct proof of the aforementioned 
result from \cite{bbmt} about provable failures of the Partition Hypothesis. 
In Section~\ref{Sec:PG}, we prove clause 3 of Theorem A. We conclude in 
Section~\ref{Sec:coding} with a brief discussion about using the partition principles 
introduced here to yield variations of the Halpern-L\"{a}uchli theorem pertaining to 
coding trees.

\medskip

Some related and partially overlapping results were recently independently obtained by 
Nedeljko Stefanovi\'{c} \cite{stefanovic}. In particular, he proves that 
the principle $\mathsf{DDF}$ holds after adding $\beth_\omega$-many Cohen reals to any model 
of $\mathsf{ZFC}$, thereby also yielding the proof of the Halpern-L\"{a}uchli theorem described 
above immediately after the statement of Theorem A. Stefanovi\'{c} also investigates these 
combinatorial principles in certain models of $\mathsf{ZF} + \neg \mathsf{AC}$, in particular 
proving that $\PG$ holds in Cohen's symmetric model (the same model used by Halpern and 
L\"{e}vy in \cite{halpern_levy} to show that $\mathsf{BPI}$ does not imply $\mathsf{AC}$).

\section{Preliminaries} \label{Sec:prelim}

\subsection{Notation}

If $a$ is a 
set of ordinals, then we will sometimes think of $a$ as the increasing function whose 
domain is the order type of $a$ (which we denote by $\otp(a)$). In particular, if $\eta < \otp(a)$, then 
$a(\eta)$ is the unique $\beta \in a$ such that $\otp(a \cap \beta) = \eta$. 
Similarly, if $I \subseteq \otp(a)$, then $a[I]$ denotes the set $\{a(\eta) \mid 
\eta \in I\}$. If $X$ is any set and $n < \omega$, then $[X]^n$ denotes the set 
of $n$-element subsets of $X$. If $X$ is a set of ordinals, then we will use 
the notation $(\alpha_0, \ldots, \alpha_{n-1}) \in [X]^n$ to denote the assertion that 
$\{\alpha_0, \ldots, \alpha_{n-1}\} \in [X]^n$ and $\alpha_0 < \ldots < \alpha_{n-1}$.

\subsection{Trees and the Halpern-L\"auchli theorem}

A \emph{tree} is a partially ordered set $T$ with the property that for every $t\in T$, the set $\pred_T(t):= \{s\in T: s<_T t\}$ of \emph{predecessors} of $t\in T$ is well-ordered. In particular, any subset $S\subseteq T$ equipped with the partial order induced from $T$ is also a tree. We write $h_T(t) := \otp(\pred_T(t))$ and call this the \emph{height} of $t\in T$. If $\alpha$ is an ordinal, we write $T(\alpha) := \{t\in T: h_T(t) = \alpha\}$ for level $\alpha$ of $T$. The \emph{height} of $T$ is the ordinal $h(T) := \sup\{h_T(t): t\in T\}$. We write $\suc_T(t) := \{u\in T: t<_T u\}$, and given $\alpha> h_T(t)$, we write $\suc_T(t, \alpha) := \{u\in \suc_T(t): h_T(u) = \alpha\}$. We write $\is_T(t) := \suc_T(t, h_T(t)+1)$ for the set of \emph{immediate successors} of $t\in T$. If $S\subseteq T$, we can also write $\suc_T(S) := \bigcup_{t\in S} \suc_T(t)$ and $\is_T(S) := \bigcup_{t\in S} \is_T(t)$, and if $\alpha > h_T(t)$ for every $t\in S$, we can also write $\suc_T(S, \alpha) := \bigcup_{t\in S} \suc_T(t, \alpha)$. We say $T$ is \emph{finitely branching} if $\is_T(t)$ is finite for every $t\in T$, and we say that $T$ has \emph{no terminal nodes} if $\is_T(t)\neq \emptyset$ for every $t\in T$. We say $T$ is \emph{perfect} if for every $t\in T$, some $u\in\suc_T(t)$ has $|\is_T(u)|\geq 2$. We say $T$ is \emph{rooted} if $|T(0)| = 1$. Unless specified otherwise, all trees in this paper will be rooted, perfect, finitely branching, height $\omega$, and without terminal nodes.

Given a tree $T$, a \emph{branch through $T$} is a maximal linearly ordered subset of $T$. We write $[T]$ for the set of branches through $T$. Every $x\in [T]$ satisfies $|x\cap T(m)| = 1$ for every $m< \omega$, so we write $x(m)\in T(m)$ for this unique element. We equip $[T]$ with the topology of pointwise convergence, where $x_i\to x$ iff for every $m< \omega$, we eventually have $x_i(m) = x(m)$. Because of our standing assumptions on trees, we have that $[T]$ is homeomorphic to Cantor space.

Now suppose $0< d< \omega$ and that $\vec{T} = \langle T_0,\ldots,T_{d-1}\rangle$ is a sequence of trees with no terminal nodes. The \emph{level product} of these trees is the set $\bigotimes \vec{T} := \bigcup_{m< \omega} \prod_{i< d} T_i(m)$. We turn $\bigotimes \vec{T}$ into a tree, where given $\vec{s}= (s_0,\ldots,s_{d-1})$ and $\vec{t} = (t_0,\ldots,t_{d-1})$ in $\bigotimes \vec{T}$, we set $\vec{s}\leq_{\bigotimes\vec{T}} \vec{t}$ iff $s_i\leq_{T_i} t_i$ for each $i< d$. To ease notation, we write $\leq_{\vec{T}}$ in place of $\leq_{\bigotimes\vec{T}}$. Note that $\left(\bigotimes \vec{T}\right)(m) = \prod_{i< d} T_i(m)$ and that as spaces, we have $[\bigotimes\vec{T}]\cong \prod_{i< d} [T_i]$. Given $\vec{x} \in \prod_{i<d} [T_i]$, if 
$i < d$, then we will always denote the $i^{\mathrm{th}}$ entry in $\vec{x}$ as 
$x_i \in [T_i]$. Moreover, if $m < \omega$, then we let $\vec{x}(m) := \langle x_0(m), x_1(m), \ldots, x_{d-1}(m) \rangle \in \prod_{i<d} T_i(m)$.

\begin{definition}
	\label{Def:StrongSubtree}
	Let $T$ be a tree with no terminal nodes, and let $a\subseteq \omega$ be infinite. An \emph{$a$-strong subtree} of $T$ is a subset $S\subseteq T$ built inductively as follows. 
	\begin{itemize}
		\item 
		Pick any $t_0\in T(a(0))$ and set $S(0) = \{t_0\}$.
		\item
		Inductively assume for some $m< \omega$ that $S(m)$ has been determined and that $S(m)\subseteq T(a(m))$. Then for every $u\in \is_T(S(m))$, pick some $t_u\in T(a(m+1))$, and set $S(m+1) = \{t_u: u\in \is_T(S(m))\}$. 
	\end{itemize} 
\end{definition}

Notice that if $S\subseteq T$ is an $a$-strong subtree, then the set $S(m)$ from the inductive construction above is in fact level $m$ of the tree $S$. Let us remark that if $\vec{T} = \langle T_0,\ldots,T_{d-1}\rangle$ is a finite sequence of trees and $S_i\subseteq T_i$ is an $a$-strong subtree for each $i< d$, then $\bigotimes \vec{S}$ is an $a$-strong subtree of $\bigotimes \vec{T}$; however, not every $a$-strong subtree of $\bigotimes\vec{T}$ has this particularly nice form.

We can now state the theorem of Halpern and L\"auchli which is the main topic of this paper.

\begin{theorem}[Halpern-L\"auchli \cite{halpern_lauchli}]
	\label{Thm:HL}
	Let $0< d< \omega$, and let $\vec{T} = \langle T_0,\ldots,T_{d-1}\rangle$ be a sequence of trees. Let $r< \omega$, and suppose $\gamma\colon \bigotimes \vec{T}\to r$ is a coloring. Then there are an infinite $a\subseteq \omega$ and $a$-strong subtrees $S_i\subseteq T_i$ so that writing $\vec{S} = \langle S_0,\ldots,S_{d-1}\rangle$, we have that  $\bigotimes \vec{S}$ is monochromatic for $\gamma$.
\end{theorem}

We write $\mathsf{HL}$ for the statement of Theorem~\ref{Thm:HL} and $\mathsf{HL}_d$ for its restriction to sequences of trees of length at most $d$. 

\section{Polish space partition principles}
\label{Sec:PGdef}

While $\mathsf{HL}_1$ is trivial, proving $\mathsf{HL}_d$ by induction on $d$ is quite difficult. Good references for proofs along these lines are \cite{milliken} and \cite{stevo_book}. However, we draw attention to a proof due to Harrington (cf.\ \cite{todorcevic_farah}) using ideas from forcing. Let us begin by giving a very brief, high-level overview of the structure of Harrington's proof. Given $\gamma\colon \bigotimes\vec{T}\to r$ as in Theorem~\ref{Thm:HL}, one considers the poset $\bb{P}$ for adding a large number of Cohen reals, which are viewed as members of $[\bigotimes\vec{T}]$. Upon fixing a name $\dot{U}$ for a non-principal ultrafilter on $\omega$ and various names $\dot{b}$ for $\bb{P}$-generic branches, certain conditions $q_{\dot{b}}\in \bb{P}$ force that for $\dot{U}$-many levels, the corresponding node of the branch $\dot{b}$ is mapped by $\gamma$ to some color $i_{\dot{b}}< r$. By using the Erd\H{o}s-Rado theorem, we can find a rich collection of such $\dot{b}$ so that various properties of the corresponding $q_{\dot{b}}$ and $i_{\dot{b}}$ are the same. One then uses the conditions $q_{\dot{b}}$ to help build the subtrees $\vec{S}$ with $\bigotimes\vec{S}$ monochromatic for $\gamma$.

It is natural to attempt to remove some of the forcing formalism from these ideas. Namely, if $U\in \beta\omega\setminus \omega$ is a non-principal ultrafilter and if $\gamma\colon \bigotimes\vec{T}\to r$ is a coloring for some $r< \omega$, we can define a coloring $\gamma_{U}\colon \prod_{i< d} [T_i]\to r$ via $\gamma_{U}(\vec{x}) = j$ iff $\{m< \omega: \gamma(\vec{x}(m)) = j\}\in U$. Of course, the coloring $\gamma_{U}$ will typically have horrible definability properties, i.e.\ fail to have the Baire property. However, we can still attempt to reason about the possible Ramsey-theoretic properties of arbitrary colorings on products of Polish spaces. This line of thought naturally leads to the definition of 
the partition principles $\PG_d$, which, recalling Definition~\ref{Def:PG}, is the assertion 
that, for every positive $d < \omega$, any sequence $\langle X_0, \ldots, X_{d-1} \rangle$ 
of perfect Polish spaces, any $r < \omega$, and any coloring $\gamma:\prod_{i<d} X_i \to r$, 
there is a somewhere dense grid in $\prod_{i<d} X_i$ that is monochromatic for $\gamma$.

%

\begin{remark} 
	\label{Rem:GDB}
	Recall that every perfect Polish space contains a dense $G_\delta$ subspace homeomorphic to Baire space ${}^\omega \omega$. Therefore to show that $\mathsf{PG}_d$ holds, one may assume that each $X_i$ is the Baire space.
\end{remark}

\begin{proposition}
	\label{Prop:PGHL}
	$\mathsf{PG}_d$ implies $\mathsf{HL}_d$.
\end{proposition}

\begin{proof}
Let $\gamma\colon \bigotimes \vec{T}\to r$ be a coloring. Fix a non-principal ultrafilter $U\in \beta\omega\setminus \omega$, and form the coloring $\gamma_{U}\colon \prod_{i< d} [T_i]\to r$ defined before Remark~\ref{Rem:GDB}. Using $\mathsf{PG}_d$, find somewhere dense sets $Y_i\subseteq [T_i]$ so that $\gamma_{U}[\prod_{i< d} Y_i] = \{j\}$ for some $j< r$. As each $Y_i$ is somewhere dense, we can find $t_i\in T_i$ so that whenever $u\geq_{T_i} t_i$, there is $y\in Y_i$ with $u\in y$. By moving some of the $t_i$ further up if needed to place them all on the same level, we may assume that $(t_0,\ldots,t_{d-1}):= \vec{t}\in \bigotimes \vec{T}$. 

We now inductively construct an infinite $a\subseteq \omega$ and $a$-strong subtrees $S_i\subseteq T_i$ with $\bigotimes\vec{S}$ monochromatic for $\gamma$. To get started, for each $i< d$, pick some $y_i\in Y_i$ with $t_i\in y_i$. Setting $\vec{y} = (y_0,\ldots,y_{d-1})\in \prod_{i< d} Y_i$, we have $\gamma_U(\vec{y}) = j$. This means that $W_0 := \{m< \omega: \gamma(\vec{y}(m)) = j\}\in U$, so in particular is infinite. Pick some $a(0)\in W_0$ above $\ell(\vec{t})$, and for each $i< d$, we set $S_i(0) = \{y_i(a(0))\}$. 

Now suppose $n> 0$ and that for every $m< n$, both $a(m) < \omega$ and $S_i(m)\subseteq T_i(a(m))$ have been determined so that every $s\in S_i$ satisfies $s\geq_{T_i} t_i$. For each $u\in \is_{T_i}(S_i(n-1))$, there is some $y_u\in Y_i$ with $u\in y_u$. Set $Z_i := \{y_u: u\in \is_{T_i}(S_i(n-1))\}$. Because $Z_i\subseteq Y_i$, we have that $W_n := \bigcap \{\{m< \omega: \gamma(\vec{z}(m)) = j\}: \vec{z}\in \prod_{i< d} Z_i\}\in U$. Pick some $a(n)\in W_n$ with $a(n)> a(n-1)$, and for each $i< d$, set $S_i(n) = \{z(n): z\in Z_i\}$.
\end{proof}

We will show in our proof of Theorem A(3) in Section~\ref{Sec:PG} that one can force that $\mathsf{PG}$ is consistent by adding $\beth_\omega$-many Cohen reals. In particular, as $\mathsf{HL}$ is a $\bf{\Pi}^1_2$ statement, Shoenfield's absoluteness theorem then implies that $\mathsf{HL}$ is true in ZFC, yielding a new proof of $\mathsf{HL}$. In a sense, this proof re-interprets Harrington's forcing proof of $\mathsf{HL}$ by actually passing to the generic extension, whereas Harrington's proof can be phrased just in terms of combinatorics on the forcing poset.

Upon analyzing the proof of Proposition~\ref{Prop:PGHL}, it becomes clear that $\mathsf{PG}$ is actually stronger than what we need to prove $\mathsf{HL}$. We define two weakenings of $\mathsf{PG}$ which are still strong enough so that the proof of Proposition~\ref{Prop:PGHL} goes through. The first one is the weakest such principle and is in fact true in $\mathsf{ZFC}$.

\begin{definition}
	\label{Def:FPG}
	If $X$ is a topological space and $\mc{U}$ is a collection of open subsets of $X$, a \emph{$\mc{U}$-set} is any $Y\subseteq X$ which meets every member of $\mc{U}$.
	
	Given $0< d< \omega$ and a sequence $\langle X_0,\ldots,X_{d-1}\rangle$ of perfect Polish spaces, a \emph{finitary somewhere dense grid} is a subset $Y\subseteq \prod_{i< d} X_i$ so that for each $i< d$, there are open $U_i\subseteq X_i$ so that for every sequence $\langle \mc{U}_0,\ldots,\mc{U}_{d-1}\rangle$  with $\mc{U}_i$ a finite collection of non-empty open subsets of $U_i$, there is for each $i< d$ a $\mc{U}_i$-set $Y_i\subseteq X_i$ with $\prod_{i< d} Y_i\subseteq Y$. 
	
	The \emph{finitary Polish grid principle} $\mathsf{FPG}_d$ states that for any sequence $\langle X_0,\ldots,X_{d-1}\rangle$ of perfect Polish spaces, any $r< \omega$, and any coloring $\gamma\colon \prod_{i< d} X_i\to r$, there is a monochromatic finitary somewhere dense grid.  
\end{definition} 

In fact, one can show in ZFC that for any $d< \omega$, any sequence $\langle X_0,\ldots,X_{d-1}\rangle$ of perfect Polish spaces, any finitary somewhere dense grid $Y\subseteq \prod_{i< d} X_i$, and any finite coloring of $Y$\!, one of the colors contains a finitary somewhere dense grid. However, we won't say much more about $\mathsf{FPG}_d$;  to us at least, it seems that proofs of $\mathsf{FPG}_d$ either go through forcing the consistency of the much stronger $\mathsf{PG}_d$ or go through repeating many of the steps seen in combinatorial proofs of $\mathsf{HL}$.  

In between $\mathsf{FPG}_d$ and $\mathsf{PG}_d$, we have our last principle, which was 
defined by the second author in \cite{hl2d}. 

\begin{definition}
	\label{Def:DDF}
	Given a sequence $\langle X_0,\ldots,X_{d-1}\rangle$ of perfect Polish spaces and $I\subseteq d$, we write $\pi_I\colon \prod_{i< d} X_i\to \prod_{i\in I} X_i$ for the natural projection map. Given some $Z\subseteq \prod_{i< d} X_i$ and some $x\in \prod_{i\in I} X_i$, we write $Z_x := \pi_{d\setminus I}[\{z\in Z: \pi_I(z) = x\}]$.  
	
	We define the notion of a set $Z\subseteq \prod_{i< d} X_i$ being a \emph{dense-by-dense filter}, DDF for short, by induction on $d> 0$. 
	\begin{enumerate}
		\item 
		$Z\subseteq X_0$ is a DDF set if it is dense. 
		\item
		Given $d> 0$, we have that $Z\subseteq \prod_{i\leq d} X_i$ is DDF if $\pi_d[Z]\subseteq \prod_{i< d} X_i$ is DDF and $\{Z_x: x\in \pi_d[Z]\}$ generates a filter of dense subsets of $X_d$.
	\end{enumerate}
	We say that $Z\subseteq \prod_{i< d} X_i$ is \emph{somewhere-DDF} if for some non-empty open sets $U_i\subseteq X_i$, $Z$ is DDF as a subset of $\prod_{i < d} U_i$.

	The principle $\mathsf{DDF}_d$ states that for any sequence $\langle X_0,\ldots,X_{d-1}\rangle$ of perfect Polish spaces, any $r< \omega$, and any coloring $\gamma\colon \prod_{i< d} X_i\to r$, there is a monochromatic somewhere-DDF subset. 
\end{definition}

Since a somewhere dense grid is clearly somewhere-DDF, it is clear that $\PG_d$ implies 
$\DDF_d$. The next proposition shows that a witness to $\DDF_d$ is also a witness to 
$\mathsf{FPG}_d$.

\begin{proposition}
	\label{Prop:DDFImpliesFPG}
	Given $0< d< \omega$ and $\langle X_0,\ldots,X_{d-1}\rangle$ a sequence of perfect Polish spaces, if $Z\subseteq \prod_{i< d} X_i$ is somewhere-DDF, then $Z$ is also a finitary somewhere dense grid.
\end{proposition}

\begin{proof}
	We prove by induction on $d$ that if $U_i\subseteq X_i$ are non-empty open sets with $Z\subseteq \prod_{i< d} U_i$ a DDF subset, then the $U_i$ also witness that $Z$ is a finitary somewhere dense grid. For $d = 1$ this is clear. Now suppose the result is known for dimension $d$. Let $\langle X_0,\ldots,X_d\rangle$ be perfect Polish spaces, and suppose $Z\subseteq \prod_{i\leq d} U_d$ is DDF. Fix $\langle \mc{U}_0,\ldots,\mc{U}_d\rangle$ with each $\mc{U}_i$ a finite collection of non-empty open subsets of $U_i$. As $\pi_d[Z]\subseteq \prod_{i< d} X_i$ is DDF, then by induction we can find $\mc{U}_i$-sets $Y_i\subseteq U_i$ with $\prod_{i<d} Y_i\subseteq Z$. Find some finite $F_i\subseteq Y_i$ which is also a $\mc{U}_i$-set. As $\{Z_x: x\in \pi_d[Z]\}$ generates a filter of dense subsets of $U_d$, the set $Y_d = \bigcap \{Z_y: y\in \prod_{i< d} F_i\}\subseteq U_d$ is dense, so in particular is a $\mc{U}_d$-set. As $F_0\times\cdots\times F_{d-1}\times Y_d\subseteq Z$, we see that $Z$ is a finitary somewhere dense grid.
\end{proof}

This principle serves as an interesting middle ground between $\mathsf{FPG}_d$ and $\mathsf{PG}_d$; it is rich enough in that one can use non-combinatorial tools to investigate it, yet weak enough that one might hope to prove it in ZFC. Indeed, we have the following.

\begin{proposition}
	\label{Prop:2DDF}
	$\mathsf{DDF}_2$ is true.
\end{proposition}

\begin{proof}
	For reasons that will be clear at the end of the proof, we prove something slightly stronger. Let $X_0$ be a non-meager subset of some perfect Polish space $\tilde{X}_0$, and let $X_1$ be a somewhere dense subset of some perfect Polish space $\tilde{X}_1$. By zooming in to the relevant open sets, we may assume that $X_0\subseteq \tilde{X}_0$ is nowhere meager and that $X_1\subseteq \tilde{X}_1$ is dense. Let $r< \omega$, and let $\gamma\colon X_0\times X_1\to r$ be a coloring. We will prove that there is a somewhere-DDF subset of $X_0 \times X_1$ that is monochromatic for $\gamma$.
	
	The proof is by induction on $r$. It is trivial if $r=1$, so assume that $r>1$. We attempt to find a DDF subset inside $Z := \gamma^{-1}(\{r-1\})$. Let $\{U_n: n< \omega\}$ be a basis for $X_0$. We inductively attempt to build a decreasing collection $\{Y_n: n< \omega\}$ of dense subsets of $X_1$ and a dense subset $\{x_n : n < \omega\}$ of $X_0$ as follows. Set $Y_0 = X_1$. If $n< \omega$ and $Y_n$ has been determined, pick any $x_n\in U_n$ such that $Z_{x_n}\cap Y_n\subseteq X_1$ is dense, and set $Y_{n+1} := Z_{x_n}\cap Y_n$. If we can do this for every $n< \omega$, then $\bigcup_{n< \omega} \{x_n\}\times Y_{n+1}\subseteq Z$ is a DDF set. 
	
	Suppose for some $k< \omega$ that we fail to construct $x_k$ and $Y_{k+1}$. Let $\{V_n: n< \omega\}$ be a basis for $X_1$. For every $x\in U_k$, there is some $n_x< \omega$ so that $Z_x\cap Y_k\cap V_{n_x} = \emptyset$. For some $n< \omega$, the set $W:= \{x\in U_k: n_x = n\}$ is non-meager. Then $W\subseteq \tilde{X}_0$ is non-meager, $Y_k\cap V_n\subseteq \tilde{X}_1$ is somewhere dense, and $\gamma$ attains one fewer color on $W\times (Y_k\cap V_n)$. We can therefore apply the induction hypothesis to $\gamma \restriction W \times (Y_k\cap V_n)$ to obtain a monochromatic somewhere-DDF subset of $W \times (Y_k \cap V_n)$, and hence also of $X_0 \times X_1$.
\end{proof}

Using similar ideas, one can also say something about the stronger $\mathsf{PG}_2$ principle.

\begin{proposition}
	\label{Prop:PG2True2Colors}
	$\mathsf{PG}_2$ restricted to $2$-colorings is true.
\end{proposition} 

\begin{proof}
	Let $X_0$ and $X_1$ be perfect Polish spaces, and let $\gamma:X_0 \times X_1 \rightarrow 
	2$ be a coloring. We will attempt to construct a dense grid inside $Z_1 := \gamma^{-1}
	(\{1\})$ and will show that if any step of the construction fails, then we can find a 
	somewhere dense grid inside $Z_0 := \gamma^{-1}(\{0\})$. Given $j < 2$ and $x \in X_0$, 
	let $(Z_j)_x := \{y \in X_1 : (x,y) \in Z_j\}$. Similarly, given $y \in X_1$, let 
	$(Z_j)^y := \{x \in X_0 : (x,y) \in Z_j\}$.
	
	For $i < 2$, let $\{U^i_n:n<\omega\}$ be a basis for $X_i$. 
	We attempt to build sets $Y_i = \{x^i_n : n < \omega\}$ for $i < 2$ such that
	$Y_0 \times Y_1 \subseteq Z_1$ and, for all $n < \omega$ and $i < 2$, we have 
	$x^i_n \in U^i_n$. During the construction, we will also construct $\subseteq$-decreasing 
	sequences $\langle V^i_n : n < \omega \rangle$ of nowhere meager subsets of $X_i$ with 
	the property that, for all $n < \omega$, both $\{x^0_m : m < n\} \times V^1_n$ and 
	$V^0_n \times \{x^1_m : m < n\}$ are subsets of $Z_1$.
	
	Begin by letting $V^i_0 = X_i$ for $i < 2$. Now suppose that $n < \omega$ and we have 
	chosen $\langle x^i_m : m < n \rangle$ and $\langle V^i_m : m \leq n \rangle$. We will 
	choose $x^0_n \in V^0_n \cap U^0_n$ and a nowhere meager set $V^1_{n+1} \subseteq 
	V^i_n$, and then we will choose $x^1_n \in V^1_{n+1} \cap U^1_n$ and a nowhere meager set 
	$V^0_{n+1} \subseteq V^0_n$.
	
	If we are able to find $x \in V^0_n \cap U^0_n$ such that $V^1_n \cap (Z_1)_x$ is 
	nowhere meager, then let $x^0_n$ be such an $x$ and let $V^1_{n+1}:= V^1_n \cap 
	(Z_1)_{x_n}$. Suppose momentarily that we were unable to find such an $x$. Then, for every 
	$x \in V^0_n \cap U^0_n$, there is $k_x < \omega$ such that 
	$V^1_n \cap (Z_1)_x \cap U^1_{k_x}$ is meager. Then there is a fixed $k < \omega$ and a 
	non-meager set $W_0 \subseteq V^0_n \cap U^0_n$ such that $k_x = k$ for all $x \in W_0$. 
	Let $W_0'$ be a countable somewhere dense subset of $W_0$. Then $V^1_n \cap U^1_k \cap 
	\bigcup\{(Z_1)_x : x \in W_0'\}$ is meager, so $W_1 := V^1_n \cap U^1_k \cap \bigcap
	\{(Z_0)_x : x \in W_0'\}$ is nonmeager. In particular, $W_1$ is somewhere dense, so 
	$W_0' \times W_1$ is a somewhere dense grid contained in $Z_0$. 
	
	We can therefore assume that we were able to construct $x^0_n$ and $V^1_{n+1}$ and 
	continue to the second half of step $n$ of the construction, where a symmetric argument 
	shows that we can either find 
	\begin{enumerate}
	  \item $y \in V^1_{n+1}$ such that $V^0_n \cap (Z_1)^y$ is nowhere meager; or
	  \item a somewhere dense grid contained in $Z_0$.
	\end{enumerate}
	We can therefore again assume we are in case (1), let $x^1_n$ be such a $y$, and let 
	$V^0_{n+1} := V^0_n \cap (Z_1)^{x^1_n}$.
	
	At the end of the construction, we have produced a dense grid $Y_0 \times Y_1$. To 
	see that it is a subset of $Z_1$, fix $m,n \leq \omega$. If $m \leq n$, then we 
	ensured that $x^1_n \in V^1_{n+1} \subseteq V^1_{m+1}$, and hence $(x^0_m,x^1_n) \in 
	Z_1$. Similarly, if $m > n$, then we ensured that $x^0_m \in V^0_m \subseteq V^0_{n+1}$, 
	so again $(x^0_m,x^1_n) \in Z_1$.
\end{proof}

It remains open whether the full $\mathsf{PG}_2$ is true in ZFC. For $d> 2$, we will show in Theorem~\ref{Thm:DDFbad} that $\mathsf{DDF}_d$, even restricted to $2$-colorings, implies that $\mathfrak{c}\geq \aleph_{d-1}$.

\section{Consistent failure of $\mathsf{DDF}_d$}
\label{Sec:DDFbad}

\begin{definition}
	Let $S$ be an infinite set, and suppose that $\mc H$ is a collection of subsets 
	of $S$. We say that $\mc H$ is \emph{weakly partition regular} if for every finite 
	partition $S = \bigcup_{j < k} P_j$ of $S$, there is $j < k$ such that 
	$P_j \in \mc H$
	
	Given an infinite cardinal $\kappa$, we say that $\mc H$ is 
	\emph{weakly $\kappa$-partition regular} if for every 
	partition $S = \bigcup_{\eta < \kappa} P_\eta$ of $S$ into $\kappa$-many parts, 
	there is $\eta < \kappa$ such that $P_\eta \in \mc H$.
	
	A further weakening of this notion will be useful for us. Given an infinite 
	regular cardinal $\kappa$, we say that $\mc H$ is \emph{weakly $\kappa$-partition 
		subregular} if for every partition $S = \bigcup_{\eta < \kappa} P_\eta$ of 
	$S$ into $\kappa$-many parts, there is $\xi < \kappa$ such that 
	$\bigcup_{\eta < \xi} P_\eta \in \mc H$.
\end{definition}

\begin{remark} \label{coloring_remark}
	It will often be more convenient to phrase these partition regularity properties in terms 
	of colorings instead of partitions. For instance, a collection $\mc H$ of subsets of 
	a set $S$ is weakly $\kappa$-partition subregular if, for every set $X$ of cardinality $\kappa$ 
	and every coloring $c:S \rightarrow X$, there is $Y \in \mc H$ such that 
	$|c[Y]| < \kappa$. Throughout, we shall interchangeably use the partition and 
	coloring formulations without explicit comment.
\end{remark}

Given $0< d< \omega$ and $\vec{X} = \langle X_0,..., X_{d-1}\rangle$ a sequence of perfect Polish spaces, let $\mathsf{DDF}(\vec{X})$ denote the set of subsets of $\prod_{i < d} X_i$ which contain a somewhere-DDF subset. So the principle $\mathsf{DDF}_d$ says that for every such $\vec{X}$, the collection $\mathsf{DDF}(\vec{X})$ is weakly partition regular. Now consider another perfect Polish space $X_d$. We want to consider how partition properties of $\mathsf{DDF}(\vec{X})$ affect those of $\mathsf{DDF}(\vec{X}^\frown X_d)$.

\begin{proposition} \label{sideways_prop}
	Fix $d \geq 2$ and let $\vec{X}$ and $X_d$ be as above. 
	If $\mathsf{DDF}(\vec{X})$ is not weakly $\omega$-partition subregular, 
	then there is a $2$-coloring of $\prod_{i\leq d} X_i$ witnessing that $\mathsf{DDF}(\vec{X}^\frown X_d)$ is not weakly partition regular.
\end{proposition}

\begin{proof}
	Assume that $\mathsf{DDF}(\vec{X})$ is not weakly $\omega$-partition 
	subregular, and fix a partition $\prod_{i < d} X_i = \bigcup_{j < \omega} P_j$ 
	such that, for all $k < \omega$, $\bigcup_{j < k} P_j$ does not contain 
	a somewhere-DDF set. 

	Let $\{S_n: n< \omega\}$ be a sequence of open subsets of $X_d$ with the following property:
	\begin{itemize}
		\item 
		For every non-empty open $U\subseteq X_d$, there is $N< \omega$ such that for all $n\geq N$, we have $S_n\cap U\neq \emptyset$ and $\mathrm{Int}(X_d\setminus S_n)\cap U\neq \emptyset$.
	\end{itemize} 
	We now describe a partition $\prod_{i \leq d} X_i = P^*_0 \cup P^*_1$ that will witness that $\mathsf{DDF}(\vec{X})$ is not weakly partition regular. Given $\vec{x} = \langle x_0, \ldots, x_d \rangle \in \prod_{i \leq d} X_i$, first 
	let $j(\vec{x})$ be the unique natural number $j$ such that $\langle x_0, \ldots, x_{d-1} \rangle \in P_j$. Now put $\vec{x}$ into $P^*_0$ if $x_d \in S_{j(\vec{x})}$ and into $P^*_1$ otherwise.
	
	We claim that this partition is as desired. Suppose for the sake of contradiction that there are non-empty open $U_i\subseteq X_i$ such that $P^*_0$ (wlog) contains a set $Y$ which is DDF in $\prod_{i\leq d} U_i$. However, let $N< \omega$ be such that for all $n\geq N$, we have that $S_n\cap U_d$ and $\mathrm{Int}(X_d\setminus S_n)\cap U_d$ are non-empty. Then for every $\vec{x} = (x_0,\ldots,x_{d})\in Y$, we must have that $(x_0,\ldots,x_{d-1})\in \bigcup_{j< N} P_j$. This contradicts the assumption that $\bigcup_{j<N} P_j$ does not contain a somewhere-DDF set. \qedhere   

\end{proof}

We now show that if the continuum is too small, then for a given $d< \omega$ and any $\vec{X}$ as above, we have that $\mathsf{DDF}(\vec{X})$ is not weakly $\omega$-partition subregular. To that end, we now define a sequence of colorings $\langle c_n : 1 \leq n < \omega \rangle$, 
where $c_n:[\omega_n]^{n+1} \rightarrow \omega$ for all $1 \leq n < \omega$. The 
definition is by recursion on $n$. First, let $c_1:[\omega_1]^2 \rightarrow \omega$ 
be any function such that, for all $\beta < \omega_1$, the fiber 
$c_1(\cdot, \beta) : \beta \rightarrow \omega$ is injective. Now suppose that 
$1 \leq n < \omega$ and we have defined $c_n$. For each $\beta < \omega_{n+1}$, 
let $e_\beta : \beta \rightarrow \omega_n$ be an injective function. Then, 
for each $\{\alpha_0, \ldots, \alpha_n, \beta\} \in [\omega_{n+1}]^{n+2}$ with 
$\alpha_0 < \ldots < \alpha_n < \beta$, set 
\[
c_{n+1}(\alpha_0, \ldots, \alpha_n, \beta) = c_n(e_\beta(\alpha_0), \ldots, 
e_\beta(\alpha_n)).
\]
(Note that the $(n+1)$-tuple $(e_\beta(\alpha_0), \ldots, e_\beta(\alpha_n))$ 
may not be increasing, but it is certainly injective.)

Now, given $1 \leq n < \omega$ and $a \in [\omega_n]^{n+1}$, we 
specify a distinguished element $a(\ast) \in a$. We do this by 
recursion on $n$. First, if $n = 1$, then we simply let 
$a(\ast) = \min(a)$. Next, if 
$n > 1$, then let $\beta = \max(a)$, let 
$a_1 := \{e_\beta(\alpha) : \alpha \in a \setminus \{\beta\}\}$, and 
let $a(\ast) = e_\beta^{-1}(a_1(\ast))$.

\begin{lemma} \label{difference_lemma}
	Suppose that $1 \leq n < \omega$, $a,b \in [\omega_n]^{n+1}$, $a(\ast) \neq 
	b(\ast)$, and $a \setminus \{a(\ast)\} = b \setminus \{b(\ast)\}$. Then 
	$c_n(a) \neq c_n(b)$.
\end{lemma}

\begin{proof}
	We proceed by induction on $n$. If $n = 1$, then there is $\beta > \max\{a(\ast), 
	b(\ast)\}$ such that $a = \{a(\ast), \beta\}$ and $b = \{b(\ast), \beta\}$. 
	Then $c_1(a) \neq c_1(b)$ follows from the fact that $c_1(\cdot, \beta)$ is injective.
	
	Next, suppose that $n > 1$. Let $\beta = \max(a) = \max(b)$, let 
	$a_1 = \{e_\beta(\alpha) : \alpha \in a \setminus \{\beta\}\}$, and let 
	$b_1 = \{e_\beta(\alpha) : \alpha \in b \setminus \{\beta\}\}$. 
	Then 
	\begin{enumerate}
		\item $a(\ast) = e_\beta^{-1}(a_1(\ast))$ and 
		$b(\ast) = e_\beta^{-1}(b_1(\ast))$;
		\item $a_1 \setminus \{a_1(\ast)\} = b_1 \setminus \{b_1(\ast)\}$;
		\item $c_n(a) = c_{n-1}(a_1)$ and $c_n(b) = c_{n-1}(b_1)$.
	\end{enumerate}
	Items (1) and (2), combined with the induction hypothesis, imply that 
	$c_{n-1}(a_1) \neq c_{n-1}(b_1)$, and then item (3) implies that $c_n(a) \neq 
	c_n(b)$.
\end{proof}

\begin{lemma} \label{coloring_lemma}
	Fix $1 \leq n < \omega$. There is a coloring $c\colon (\omega_n)^{n+1} \rightarrow 
	(n+2) \times \omega$ and a sequence of natural numbers $\langle m_{k} : k < \omega \rangle$ 
	such that, for every $k < \omega$ and every sequence $\langle A_i : i \leq n \rangle$ 
	of elements of $[\omega_n]^{m_{k}}$, we have 
	\[
	\left|c\left[\prod_{i \leq n} A_i\right]\right| > k.
	\]
\end{lemma}

\begin{proof}
	We first define $c\colon (\omega_n)^{n+1} \rightarrow (n+2) \times \omega$. Suppose that 
	$\vec{\alpha} = \langle \alpha_i : i \leq n \rangle \in (\omega_n)^{n+1}$. 
	If there are $i < j \leq n$ such that $\alpha_i = \alpha_j$, then let $c(\vec{\alpha}) := 
	(n+1, 0)$. Otherwise, set $a_{\vec{\alpha}} := \{\alpha_i : i \leq n\}$, and note that 
	$a_{\vec{\alpha}} \in [\omega_n]^{n+1}$. Let $i_{\vec{\alpha}}$ be the unique $i \leq n$ such that 
	$a_{\vec{\alpha}}(\ast) = \alpha_i$, and let 
	$c(\vec{\alpha}) = (i_{\vec{\alpha}}, c_n(a_{\vec{\alpha}}))$.
	
	We now define $\langle m_{k} : k < \omega \rangle$. First, let $m_{0} = 1$. 
	If $k > 0$, then first let $m^*_{k} < \omega$ be large enough so that 
	\[
	m^*_{k} \rightarrow (n+2)^{n+1}_k,
	\]
	i.e., for every coloring $r\colon [m^*_{k}]^{n+1} \rightarrow k$, there is $H \in [m^*_{k}]^{n+2}$ such 
	that $r{\restriction}[H]^{n+1}$ is constant. Then let $m_{k} = (n+1) \cdot m^*_{k}$.\footnote{We are 
		not making any real attempt to optimize the value of $m_{k}$ here; we 
		just need to know that a natural number $m_{k}$ with the desired properties exists.} 
	
	We claim that $\langle m_{k} : k < \omega \rangle$ is as desired. If $k = 0$, then this is clear, 
	so fix $k > 0$ and a sequence 
	$\langle A_i : i \leq n \rangle$ of elements of $[\omega_n]^{m_{k}}$, and suppose for sake 
	of contradiction that $\left|c\left[\prod_{i \leq n} A_i\right]\right| \leq k$. 
	First, since $m_{k} = (n+1) \cdot m^*_{k}$, we can find a sequence of 
	pairwise disjoint sets $\langle A^*_i : i \leq n \rangle$ such that, for all $i \leq n$, 
	$A^*_i \in [A_i]^{m^*_{k}}$. In particular, for all $\vec{\alpha} \in \prod_{i \leq n} A^*_i$, 
	$\vec{\alpha}$ is injective and hence, if $c(\vec{\alpha}) = (i,\ell)$, then $i \leq n$ and 
	$\ell = c_n(a_{\vec{\alpha}})$.
	
	For each $i \leq n$, enumerate $A^*_i$ in increasing order as $\langle \alpha^i_j \mid j < m^*_{k} 
	\rangle$, and define a coloring $r$ of $[m^*_{k}]^{n+1}$ as follows. Given 
	$u \in [m^*_{k}]^{n+1}$, let $\vec{\alpha}^*_u := \langle \alpha^i_{u(i)} \mid i \leq n \rangle$, 
	and let $r(u) = c(\vec{\alpha}^*_u)$. Since $\left|c\left[\prod_{i \leq n} A_i\right]\right| \leq k$, 
	the coloring $r$ takes at most $k$-many colors. Therefore, by our choice of $m^*_{k}$, we can 
	find $H \in [m^*_{k}]^{n+2}$ such that $r \restriction [H]^{n+1}$ is constant, say with 
	value $(i^*, \ell^*)$. By the last sentence of the previous paragraph, we know that $i^* \leq n$. 
	
	Enumerate $H$ in increasing order as $\langle \ell_0, \ldots, \ell_{i^* - 1}, \ell_{i^*, 0}, 
	\ell_{i^*, 1}, \ell_{i^* + 1}, \ldots, \ell_n \rangle$, i.e.:
	\begin{itemize}
		\item the first $i^*$-many elements of $H$ are $\langle \ell_0, \ldots, \ell_{i^* - 1} \rangle$;
		\item the next $2$ elements of $H$ are $\langle \ell_{i^*, 0}, \ell_{i^*, 1} \rangle$;
		\item the final $(n-i^*)$-many elements of $H$ are $\langle \ell_{i^* + 1}, \ldots, \ell_n \rangle$.
	\end{itemize}
	Let $u_0 = H \setminus \{\ell_{i^*, 1}\}$ and $u_1 = H \setminus \{\ell_{i^*, 0}\}$, let
	$\vec{\alpha}^0 = \vec{\alpha}^*_{u_0}$ and $\vec{\alpha}^1 = \vec{\alpha}^*_{u_1}$, and let 
	$a^0 = a_{\vec{\alpha}^0}$ and $a^1 = a_{\vec{\alpha}^1}$. Since $r(u_0) = r(u_1) = (i^*, \ell^*)$, 
	we have $c(\vec{\alpha}^0) = c(\vec{\alpha}^1) = (i^*, \ell^*)$, and hence 
	\begin{itemize}
		\item $a_0(\ast) = \alpha^{i^*}_{\ell^*_0}$ and $a_1(\ast) = \alpha^{i^*}_{\ell^*_1}$;
		\item $c_n(a_0) = c_n(a_1) = \ell^*$.
	\end{itemize}
	However, we also know that $a_0(\ast) \neq a_1(\ast)$ and $a_0 \setminus \{a_0(\ast)\} = 
	a_1 \setminus \{a_1(\ast)\}$ and hence, by Lemma \ref{difference_lemma}, 
	$c_n(a_0) \neq c_n(a_1)$. This is a contradiction.
\end{proof}

\begin{theorem} \label{Thm:DDFbad}
	Suppose that $2 \leq d < \omega$ and $2^{\aleph_0} \leq \aleph_{d-1}$. Then, for 
	every $d$-sequence $\vec{X}$ of perfect Polish spaces, $\mathsf{DDF}(\vec{X})$ is 
	not weakly $\omega$-partition subregular.
\end{theorem}

\begin{proof}
	Proceed by induction on $d$. Fix a sequence $\vec{X} = \langle X_0, \ldots, X_{d-1} 
	\rangle$ of perfect Polish spaces. We can assume that we in fact have $2^{\aleph_0} 
	= \aleph_{d-1}$, since if $2^{\aleph_0} = \aleph_m < \aleph_{d-1}$, then the induction 
	hypothesis will imply that $\mathsf{DDF}(\vec{X}{\restriction}(m+1))$ is not 
	weakly $\omega$-partition subregular, which immediately implies that 
	$\mathsf{DDF}(\vec{X})$ is not weakly $\omega$-partition subregular either. 
	For all $i < d$, $X_i$ is a perfect Polish space, so we can injectively enumerate $X_i$ as 
	$\langle x_{i, \alpha} : \alpha < \omega_{d-1} \rangle$.
	
	By Lemma \ref{coloring_lemma} with $n = d-1$, we can find a coloring $c \colon \prod_{i < d} X_i 
	\rightarrow (d+1) \times \omega$ and a sequence of natural numbers $\langle m_{k} 
	: k < \omega \rangle$ such that, for every $k < \omega$ and every sequence 
	$\langle A_i : i< d \rangle$ such that $A_i$ is a subset of $X_i$ of size 
	$m_{k}$ for each $i <d$, we have $|c[\prod_{i <d} A_i]| > k$. By Proposition~\ref{Prop:DDFImpliesFPG}, every $Y\in \mathsf{DDF}(\vec{X})$ contains arbitrarily large finite products, which by the discussion above implies that $c[Y]$ is infinite.
\end{proof}

\begin{corollary}
    \label{Cor:DDFbad}
	Suppose that $2 \leq d < \omega$ and $2^{\aleph_0} \leq \aleph_{d-1}$. Then, for every 
	$(d+1)$-sequence $\vec{X}$ of perfect Polish spaces, there is a $2$-coloring of $\prod_{i < d} X_i$ witnessing that $\mathsf{DDF}(\vec{X})$ is not weakly partition regular.
\end{corollary}

\begin{proof}
	This is immediate from Proposition \ref{sideways_prop} and Theorem \ref{Thm:DDFbad}.
\end{proof}

We have now established clauses (1) and (2) of Theorem A.

\begin{corollary}
  Suppose that $2 \leq d < \omega$.
  \begin{enumerate}
    \item If $\DDF_d$ holds, then $2^{\aleph_0} \geq \aleph_{d-1}$.
    \item If $\PG_d(\aleph_0)$ holds, then $2^{\aleph_0} \geq \aleph_d$.
  \end{enumerate}
\end{corollary}

\begin{proof}
  Note that $\DDF_d$ is equivalent to the assertion that, for every 
  $d$-sequence $\vec{X}$ of perfect Polish spaces, $\DDF(\vec{X})$ is weakly partition regular, 
  and $\PG_d(\aleph_0)$ implies that $\DDF(\vec{X})$ is weakly $\omega$-partition 
  subregular. Clauses (1) and (2) then immediately follow from Corollary~\ref{Cor:DDFbad} and 
  Theorem~\ref{Thm:DDFbad}, respectively.
\end{proof}

\subsection{On the Partition Hypothesis} In \cite{bbmt}, Bannister, Bergfalk, Moore, and 
Todor\v{c}evi\'{c} introduce a partition hypothesis denoted $\mathrm{PH}_n(\Lambda)$, where 
$n < \omega$ is a dimensional parameter and $\Lambda$ is an arbitrary directed 
quasi-order. They prove there that, for all $n < \omega$, $\mathrm{PH}_n(\omega_n)$ fails, 
where $\omega_n$ has the usual ordinal ordering. The proof of this fact presented in 
\cite{bbmt} makes heavy use of ideas coming from simplicial homology. Here, we show how the 
results from this section yield a direct, purely combinatorial proof. We first recall the 
following definitions from \cite{bbmt}.

\begin{definition}
  Suppose that $1 \leq n < \omega$ and $\Lambda$ is a directed quasi-order.
  \begin{enumerate}
    \item If $\vec{x}, \vec{y} \in \Lambda^{\leq n}$, then we write $\vec{x} 
    \trianglelefteq \vec{y}$ to indicate that $\vec{x}$ is a subsequence of 
    $\vec{y}$ (not necessarily an initial segment). $\vec{x} \triangleleft \vec{y}$ 
    indicates that $\vec{x}$ is a \emph{proper} subsequence of $\vec{y}$.
    \item A function $F:\Lambda^{\leq n} \rightarrow \Lambda$ is \emph{$n$-cofinal} if
    \begin{enumerate}
      \item $x \leq F(\langle x \rangle)$ for all $x \in \Lambda$;
      \item $F(\vec{x}) \leq F(\vec{y})$ for all $\vec{x} \trianglelefteq \vec{y}$ 
      in $\Lambda^{\leq n}$.
    \end{enumerate}
    \item Let $\Lambda^{\llbracket n \rrbracket} \subseteq \prod_{i<n} \Lambda^{i+1}$ 
      consist of all $\sigma \in \prod_{i<n} \Lambda^{i+1}$ that are 
      $\trianglelefteq$-increasing. If $F:\Lambda^{\leq n} \rightarrow \Lambda$ 
      is $n$-cofinal, define $F^*:\Lambda^{\llbracket n \rrbracket} \rightarrow 
      \Lambda^n$ by letting $F^*(\sigma) = F \circ \sigma = \langle F(\sigma(i)) \mid 
      i < n \rangle$ for all $\sigma \in \Lambda^{\llbracket n \rrbracket}$.
  \end{enumerate}
\end{definition}

\begin{definition}
  Suppose that $n < \omega$ and $\Lambda$ is a directed quasi-order. The Partition 
  Hypothesis $\mathrm{PH}_n(\Lambda)$ is the following assertion: for all 
  $c:\Lambda^{n+1} \rightarrow \omega$, there is an $(n+1)$-cofinal function 
  $F:\Lambda^{\leq n+1} \rightarrow \Lambda$ such that $c \circ F^* : 
  \Lambda^{\llbracket n+1 \rrbracket} \rightarrow \omega$ is constant.
\end{definition}

We are now ready to give a direct proof of the aforementioned result from 
\cite{bbmt}.

\begin{theorem}
  For all $n < \omega$, $\mathrm{PH}_n(\omega_n)$ fails.
\end{theorem}

\begin{proof}
  As noted in \cite{bbmt}, the identity function $c:\omega \rightarrow \omega$ readily 
  witnesses the failure of $\mathrm{PH}_0(\omega)$. Therefore, fix 
  $1 \leq n < \omega$, and let $c\colon (\omega_n)^{n+1} \rightarrow (n+2) \times \omega$ 
  be the function defined in the proof of Lemma~\ref{coloring_lemma}. 
  
  We claim that $c$ witnesses the failure of $\mathrm{PH}_n(\omega_n)$. Towards a 
  contradiction, suppose that $F\colon (\omega_n)^{n+1} \rightarrow \omega_n$ is an 
  $(n+1)$-cofinal function such that $c \circ F^*$ is constant, taking value 
  $(i^*,k^*) \in (n+2) \times \omega$. As shown in \cite[Lemma 7.8]{bbmt}, we may 
  assume that $F$ is \emph{strictly increasing}, i.e., $F(\vec{x}) \lneq F(\vec{y})$ 
  for all $\vec{x} \triangleleft \vec{y}$ in $(\omega_n)^{\leq n+1}$. In particular, 
  we can assume that $F^*(\sigma) = F \circ \sigma$ is \emph{injective} for all 
  $\sigma \in \omega_n^{\llbracket n+1 \rrbracket}$, and hence, recalling the 
  definition of $c$, we know that $i^* \neq n+1$. In addition, since 
  $F \circ \sigma$ is strictly increasing for all $\sigma \in \omega_n^{\llbracket 
  n+1 \rrbracket}$, the definition of $c$ implies that we in fact have $i^* < n$.
  
  Let $\alpha^* := F(\langle 0,1, \ldots, i^* \rangle) + 1$. Now define 
  $\sigma_0, \sigma_1 \in \omega_n^{\llbracket n+1 \rrbracket}$ as follows:
  \begin{itemize}
    \item for all $i < i^*$, $\sigma_0(i) = \sigma_1(i) := \langle 0,1, \ldots, i \rangle$;
    \item $\sigma_0(i^*) := \langle 0,1,\ldots,i^* \rangle$;
    \item $\sigma_1(i^*) := \langle 0,1,\ldots,i^*-1,\alpha^*\rangle$;
    \item for all $\ell < n - i^*$, $\sigma_0(i^*+\ell+1) = \sigma_1(i^* + \ell + 1) := 
    \langle 0,1,\ldots,i^*,\alpha^*, \alpha^*+1, \ldots, \alpha^* + \ell \rangle$. 
  \end{itemize}
  Let $a_0:=F^*(\sigma_0)$ and $a_1:=F^*(\sigma_1)$. Since $\sigma_0$ and $\sigma_1$ 
  only differ in their $i^*$-th entry and $F^*(\sigma_0)$ and $F^*(\sigma_1)$ are 
  strictly increasing, we have $a_0(i) = a_1(i)$ for all $i \in (n+1) 
  \setminus \{i^*\}$. By the definition of $c$ and the fact that $c \circ F^*$ is constant 
  with value $(i^*, k^*)$, it follows that $a_0(\ast) = a_1(\ast) = i^*$ (recall the 
  notation from the paragraph preceding Lemma~\ref{difference_lemma}). 
  Moreover, since $\alpha^* > F(\sigma_0(i^*))$ and $F$ is $(n+1)$-cofinal, 
  we know that 
  \[
    a_0(i^*) = F(\sigma_0(i^*)) < \alpha^* \leq F(\sigma_1(i^*)) = a_1(i^*).
  \] 
  Then Lemma~\ref{difference_lemma} implies that $c_n(a_0) \neq c_n(a_1)$. However, 
  again by the fact that $c \circ F^*$ is constant with value $(i^*, k^*)$, it must be 
  the case that $c_n(a_0) = c_n(a_1) = k^*$, which is a contradiction.
\end{proof}

\section{Forcing $\mathsf{PG}_d$}
\label{Sec:PG}

We now finish the proof of Theorem~A by proving part $(3)$. We show that adding $\beth_{d-1}^+$-many Cohen reals to any model of ZFC yields a model of $\mathsf{PG}_d(\aleph_0)$. In particular, by starting with a model of GCH, Theorem \ref{Thm:DDFbad} is consistently sharp. We will need the notion of \emph{uniform $n$-dimensional $\Delta$-system} isolated 
in \cite{higher_delta_systems}. We begin by recalling the relevant definitions.

\begin{definition} \label{aligned_def}
	Suppose that $a$ and $b$ are sets of ordinals.
	\begin{enumerate}
		\item We say that $a$ and $b$ are \emph{aligned} if $\otp(a) = \otp(b)$ and $\otp(a \cap \gamma) = \otp(b \cap \gamma)$ for all
		$\gamma \in a \cap b$.
		In other words, if $\gamma$ is a common element of two aligned sets $a$ and $b$, then it
		occupies the same relative position in both $a$ and $b$.
		\item If $a$ and $b$ are aligned then we let $\mb{r}(a,b) :=
		\{i < \otp(a) : a(i) = b(i)\}$. Notice that, in this case,
		$a \cap b = a[\mb{r}(a,b)] = b[\mb{r}(a,b)]$.
	\end{enumerate}
\end{definition}

\begin{definition} \label{delta_system_def}
	Suppose that $H$ is a set of ordinals, $n$ is a positive integer, and $u_b$ is a set of ordinals
	for all $b \in [H]^n$.
	We call $\langle u_b : b \in [H]^n \rangle$ a \emph{uniform
		$n$-dimensional $\Delta$-system} if there are an ordinal $\rho$ and, for
	each $\mb{m} \subseteq n$, a set $\mb{r}_{\mb{m}} \subseteq \rho$
	satisfying the following statements.
	\begin{enumerate}
		\item $\otp(u_b) = \rho$ for all $b \in [H]^n$.
		\item For all $a,b \in [H]^n$ and $\mb{m} \subseteq n$, if $a$ and $b$ are aligned with $\mb{r}(a,b) = \mb{m}$,
		then $u_a$ and $u_b$ are aligned with $\mb{r}(u_a, u_b) = \mb{r}_{\mb{m}}$.
		\item For all $\mb{m}_0, \mb{m}_1 \subseteq n$, we have
		$\mb{r}_{\mb{m}_0 \cap \mb{m}_1} = \mb{r}_{\mb{m}_0} \cap \mb{r}_{\mb{m}_1}$.
	\end{enumerate}
\end{definition}

The following is a corollary of the main lemma of \cite{higher_delta_systems}.

\begin{corollary}\cite[Corollary 3.16]{higher_delta_systems} \label{aleph_1_cor}
	Suppose that $1 \leq n < \omega$, and let $\mu = \beth_{n-1}^+$.
	If $\left\langle u_b : b \in [\mu]^n \right\rangle$ is a sequence of finite
	sets of ordinals and $g:[\mu]^n \rightarrow \omega$ is a function, then there is
	$H \in [\mu]^{\aleph_1}$ such that $\left\langle u_b : b \in [H]^n \right\rangle$
	is a uniform $n$-dimensional $\Delta$-system and $g \restriction [H]^n$ is
	constant.
\end{corollary}

We now turn to the proof of Theorem A(3). Recall that every perfect Polish space contains a dense $G_\delta$ subspace homeomorphic to Baire space ${}^\omega \omega$. Therefore to show that $\mathsf{PG}_d(\aleph_0)$ holds, one may assume that each $X_i$ is the Baire space. It will be helpful to write $X_i = [T_i]$, where each $T_i$ is a copy of the tree ${}^{{<}\omega}\omega$.
Let $\theta := \beth_{d-1}^+$, and let $\bb{P} = \mathrm{Add}(\omega, \beth_{d-1}^+)$ be the forcing to add $\beth_{d-1}^+$-many Cohen reals. We think of the conditions in $\bb{P}$ as being all
finite partial functions $p:\theta \rightarrow \prod_{i < d} T_i$, and 
$q \leq_{\bb{P}} p$ if and only if $\dom(q) \supseteq \dom(p)$ and, for 
all $\alpha \in \dom(p)$ and $i < d$, we have $p(\alpha)(i) \leq_{T_i} 
q(\alpha)(i)$. For each $\alpha < \theta$ and $i \leq n$, let $\dot{x}^\alpha_i$ 
be the canonical $\bb{P}$-name for $\{p(\alpha)(i) \mid p \in \dot{G}\}$, where 
$\dot{G}$ is the canonical $\bb{P}$-name for the generic filter. By standard 
arguments, $\dot{x}^\alpha_i$ is forced to be an element of $[T_i]$ (as defined 
in $V^\bb{P}$).

For each $p \in \bb{P}$, we define a ``collapsed" version of $p$, denoted $\bar{p}$, 
as follows. Let $\ell := |\dom(p)|$, and enumerate $\dom(p)$ in increasing order 
as $\langle \alpha_k : k < \ell \rangle$. Then define $\bar{p} : \ell \rightarrow 
\prod_{i < d} T_i$ by letting $\bar{p}(k) = p(\alpha_k)$ for all $k < \ell$. 
Note that each collapsed condition is a function from some natural number to a 
countable set, so there are only countably many such collapsed conditions.

Let $\dot{c}_0$ be a $\bb{P}$-name for a function from $\prod_{i < d} [T_i] 
\rightarrow \omega$. We will really only be interested in the values of $\dot{c}_0$ on 
$(n+1)$-tuples of the generic branches $\dot{x}^\alpha_i$, so let $\dot{c}$ be a $\bb{P}$-name for a 
function from $[\theta]^{d}$ to $\omega$ defined in the following way: for all 
$(\alpha_0, \ldots, \alpha_{d-1}) \in [\theta]^{d}$, let $\dot{c}(\alpha_0, \ldots, 
\alpha_{d-1}) = \dot{c}_0(\dot{x}^{\alpha_0}_0, \dot{x}^{\alpha_1}_1, \ldots, 
\dot{x}^{\alpha_{d-1}}_{d-1})$ (recall our convention that the notation $(\alpha_0, \ldots, 
\alpha_{d-1}) \in [\theta]^{d}$ implies that $\alpha_0 < \alpha_1 < \ldots < \alpha_{d-1}$).

Fix an arbitrary $p \in \bb{P}$. We will find $q \leq_\bb{P} p$ and $j < \omega$ such 
that $q$ forces the existence of a sequence $\langle Y_i \mid i < d \rangle$ 
such that each $Y_i$ is a somewhere dense subset of $[T_i]$ and $c \restriction 
\prod_{i < d} Y_i$ is constant, taking value $j$.

For each $a\in [\theta]^{d}$, find a condition $q_a
\leq p$ and a $j_a < \omega$ such that $q_a \Vdash_{\bb{P}} ``\dot{c}(a) = j_a"$. 
Let $u_a := \dom(q_a)$. Without loss of generality, assume that $a \subseteq 
u_a$ for every $a \in [\theta]^{d}$.

By Corollary \ref{aleph_1_cor}, we can find $H \in [\theta]^{\aleph_1}$, 
a ``collapsed" condition $\bar{q}^*$, natural numbers $j^*$ and $\rho$, and 
a set $\mb{r}^* \in [\rho]^{d}$ such that
\begin{itemize}
	\item for all $a \in [H]^{d}$, we have $\bar{q}_a = \bar{q}^*$ and $j_a = j^*$
	\item $\langle u_a : a \in [H]^{d} \rangle$ forms a 
	uniform $d$-dimensional $\Delta$-system; and
	\item for all $a \in [H]^{d}$, we have $|u_a| = \rho$ and $a = u_a[\mb{r}^*]$.
\end{itemize}
By taking an initial segment if necessary, assume that $\otp(H) = \omega_1$.
Let $\langle \mb{r}_{\mb{m}} : \mb{m} \subseteq d \rangle$ 
witness the fact that $\langle u_a : a \in [H]^{d} \rangle$ is 
a uniform $d$-dimensional $\Delta$-system. For each $m < d$ and each 
$a \in [H]^m$, define $u_a$ and $q_a$ by choosing any $b \in [H]^{d}$ for 
which $b[m] = a$ and setting $u_a := u_b[\mb{r}_m]$ and $q_a := q_b \restriction 
u_a$. By our uniformization of $H$ (cf.\ \cite[Lemma 2.3]{svhdlwlc}), 
these definitions are independent of our choice of $b$. 

Let $q := q_\emptyset$. Since $q_a \leq_{\bb{P}} p$ for every $a \in [H]^{d}$, it follows 
that $\dom(p) \subseteq u_\emptyset$ and hence $q_\emptyset \leq p$.
Also, for each $i < d$, let $s_i := \bar{q}(\mb{r}^*(i))(i)$. 
In other words, $s_i \in T_i$ is such that, for all $a = (\alpha_0, \ldots, \alpha_{d-1})$ 
in $[H]^{d}$, we have $q_a(\alpha_i)(i) = s_i$. We claim that $q$ is as desired; 
in particular, $q$ forces the existence 
of a sequence of sets $\langle Y_i : i <d \rangle$ such that 
\begin{itemize}
	\item for all $i < d$, $Y_i$ is $s_i$-dense in $[T_i]$, i.e., for all $t \geq_{T_i} s_i$, 
	there is $y \in Y_i$ such that $t \in y$;
	\item $\dot{c}_0 \restriction \prod_{i < d} Y_i$ is constant, taking value $j^*$.
\end{itemize}

\begin{claim} \label{predense_claim}
	Suppose that $m < d$, $a \in [H]^m$, and $\gamma \in H \setminus (\max(a)+1)$. 
	Then the set $D_{a, \gamma} := \{ q_{a ^\frown \langle \beta \rangle} : 
	\beta \in H \setminus \gamma \}$ is pre-dense below $q_a$ in $\bb{P}$.
\end{claim}

\begin{proof}
	Let $r \leq_{\bb{P}} q_a$ be arbitrary. The set $\{u_{a ^\frown \langle \beta \rangle} 
	: \beta \in H \setminus \gamma\}$ forms an uncountable $\Delta$-system with 
	root $u_a$. We can therefore find $\beta \in H \setminus \gamma$ for which 
	$\dom(r) \cap (u_{a ^\frown \langle \beta \rangle} \setminus u_a) = \emptyset$. 
	We also know that $q_{a ^\frown \langle \beta \rangle} \restriction u_a = q_a$, 
	and $r \leq_{\bb{P}} q_a$. It follows that $r$ and 
	$q_{a ^\frown \langle \beta \rangle}$ are compatible in $\bb{P}$, so $D_{a, \gamma}$ 
	is indeed pre-dense below $q_a$.
\end{proof}

Now let $G$ be $\bb{P}$-generic over $V$ with $q \in G$, 
and let $c$ be the realization of $\dot{c}$ in 
$V[G]$. By recursively applying Claim \ref{predense_claim} $d$-many times, 
we can find a set $\delta \in [H]^{d}$, enumerated in increasing order as 
$\langle \delta_0, \ldots, \delta_{d-1} \rangle$ such that 
\begin{itemize}
	\item for all $i < d-1$, $H \cap (\delta_i, \delta_{i+1})$ is infinite;
	\item $H \cap \delta_0$ is infinite;
	\item $q_\delta \in G$.
\end{itemize}
Let $H_0$ denote the set of the first $\omega$-many elements of $H$, and for 
$i < d-1$, let $H_{i+1}$ denote the set of the first $\omega$-many elements of 
$H \cap (\delta_i, \delta_{i+1})$. Note that each $H_i$ is an element of $V$.
Now, working in $V[G]$, we will recursively construct a matrix of ordinals 
$\langle \alpha_{i, k} : i < d, ~ k < \omega \rangle$ such that, setting 
$A_i := \{\alpha_{i,k} : k < \omega \rangle$ and 
$Y_i := \{x^\alpha_i : \alpha \in A_i\}$ for all $i < d$, we have the 
following:
\begin{itemize}
	\item for all $i < d$, $A_i \subseteq H_i \cup \{\delta_i\}$;
	\item for all $i < d$, $Y_i$ is $s_i$-dense in $[T_i]$;
	\item for all $a \in \prod_{i < d} A_i$, we have $q_a \in G$, and hence 
	$c \restriction \prod_{i < d} Y_i$ is constant, taking value $j^*$.
\end{itemize}
The construction of the matrix of ordinals is by recursion on $k < \omega$ and, 
for fixed $k$, by recursion on $i < d$; in other words, the construction 
is by recursion on the anti-lexicographic ordering of $d \times \omega$. 

For each pair $(i, k) \in d \times \omega$ and each $j < d$, let 
$A_j \restriction (i,k)$ be the portion of $A_j$ constructed before stage 
$(i,k)$ of the process, i.e., $A_j \restriction (i,k) = \{\alpha_{j,\ell} : 
\ell \leq k\}$ if $j < i$ and $A_j \restriction (i,k) = \{\alpha_{j, \ell} :
\ell < k\}$ if $j \geq i$. Our recursion hypothesis will be the assumption 
that $q_a \in G$ for all $a \in \prod_{j \leq k} A_j \restriction (i,k)$ by 
the time we have reached stage $(i,k)$ of the construction. Enumerate 
${^{<\omega}\omega}$ as $\langle t_k \mid k < \omega \rangle$, with 
$t_0 = \emptyset$. We will also 
maintain the requirement that, for all $(i,k) \in d \times \omega$, 
$x^{\alpha_{i,k}}_i$ extends $s_i{}^\frown t_k$; this is what will ensure that 
$Y_i$ is $s_i$-dense in $[T_i]$.

Begin by setting $\alpha_{i, 0} := \delta_i$ for all $i \leq n$. The fact that 
$q_\delta \in G$ ensures that this satisfies the recursion hypotheses. Now suppose that 
$(i,k) \in d \times \omega$, with $i \geq 1$, and we have reached stage 
$(i,k)$ of the construction. Let 
\[
r_{i,k} = \bigcup \{q_a : a \in \prod_{j < d} A_j \restriction (i,k)\}.
\]
By our recursion hypothesis, $r_{i,k}$ is a condition in $\bb{P}$ and is in fact 
in $G$. Let $B_0 := \prod_{j < i}A_j \restriction (i,k)$ 
and $B_1 := \prod_{i < j < d} A_j \restriction (i,k)$. Note that both 
$B_0$ and $B_1$ are in $V$, as they are finite sets of finite sequences of 
ordinals. For each $\alpha \in H_i$, let 
\[
q_\alpha^* := \bigcup \{q_{b_0{}^\frown \langle \alpha \rangle ^\frown b_1} 
: b_0 \in B_0, ~ b_1 \in B_1\}.
\]
Notice that, for all $b_0, b'_0 \in B_0$ and $b_1, b'_1 \in B_1$, 
$b_0{}^\frown \langle \alpha \rangle ^\frown b_1$ and $b'_0{}^\frown \langle \alpha 
\rangle ^\frown b'_1$ are aligned; it follows that $q^*_\alpha$ is a condition in 
$\bb{P}$. Moreover, for all $(b_0, b_1) \in B_0 \times B_1$, we have 
$q_{b_0{}^\frown \langle \alpha \rangle ^\frown b_1}(\alpha)(i) = s_i$, and 
hence $q^*_\alpha(\alpha)(i) = s_i$. We can therefore extend $q^*_\alpha$ to a 
condition $q^{**}_\alpha$ with the same domain by letting $q^{**}_\alpha(\alpha)(i) 
= s_i{}^\frown t_k$ and $q^{**}_\alpha(\eta)(j) = q^*_\alpha(\eta)(j)$ for all 
$(\eta, j) \in (\dom(q^*_\alpha) \times d) \setminus \{(\alpha, i)\}$.
\begin{claim}
	The set $E := \{q^{**}_\alpha : \alpha \in H_i \setminus \{\alpha_{i, \ell} \mid 
	\ell < k\}\}$ is predense in $\bb{P}$ below $r_{i,k}$.
\end{claim}

\begin{proof}
	Fix an arbitrary condition $r \leq_{\bb{P}} r_{i,k}$; we will find a condition 
	in $E$ that is compatible with $r$. Let $\mb{m} := d \setminus \{i\}$, and 
	let $H^* := H_i \setminus \{\alpha_{i, \ell} \mid \ell < k\}$. 
	For each $(b_0, b_1) \in B_0 \times B_1$, the set $\{u_{b_0{}^\frown 
		\langle \alpha \rangle ^\frown b_1} : \alpha \in H^*\}$ forms a $\Delta$-system 
	whose root is equal to $v_{b_0, b_1} := u_{b_0{}^\frown \langle \alpha \rangle 
		^\frown b_1}[\mb{r}_{\mb{m}}]$ for some (and therefore all) $\alpha \in H^*$.
	Since there are only finitely many such pairs $(b_0, b_1)$ and since 
	$H^*$ is infinite, we can therefore fix $\alpha \in H^*$ such that, for all 
	$(b_0, b_1) \in B_0 \times B_1$, we have 
	\[
	(u_{b_0{}^\frown \langle \alpha \rangle ^\frown b_1} \setminus 
	v_{b_0, b_1})
	\cap \dom(r) = \emptyset.
	\]
	In particular, we have $\alpha \notin \dom(r)$.
	
	We claim that $q^{**}_\alpha$ and $r$ are compatible. Since $\alpha \notin \dom(r)$, 
	it suffices to show that $q_{b_0{}^\frown \langle \alpha \rangle ^\frown b_1}$ and 
	$r$ are compatible for all $(b_0, b_1) \in B_0 \times B_1$. Thus, fix such a 
	$(b_0, b_1)$. Note that $b_0{}^\frown \langle \alpha \rangle ^\frown b_1$ and 
	$b_0{}^\frown \langle \delta_i \rangle ^\frown b_1$ are aligned and that 
	$q_{b_0{}^\frown \langle \alpha \rangle ^\frown b_1} \restriction v_{b_0, b_1} 
	= q_{b_0{}^\frown \langle \delta_k \rangle ^\frown b_1} \restriction v_{b_0, b_1}$. 
	Since $r_{i,k} \leq_{\bb{P}} q_{b_0{}^\frown \langle \delta_k \rangle ^\frown b_1}$ 
	and $r \leq_{\bb{P}} r_{i,k}$, we know that $r$ is compatible with 
	$q_{b_0{}^\frown \langle \alpha \rangle ^\frown b_1} \restriction v_{b_0, b_1}$, 
	and since $u_{b_0{}^\frown \langle \alpha \rangle ^\frown b_1} \setminus v_{b_0, 
		b_1}$ is disjoint from $\dom(r)$, it follows that $r$ is compatible with 
	$q_{b_0{}^\frown \langle \alpha \rangle ^\frown b_1}$ and therefore with 
	$q^{**}_\alpha$.
\end{proof}
By the claim and the fact that $r_{i,k} \in G$,
we can fix an $\alpha_{i, k} \in H_i \setminus \{\alpha_{i,\ell} : \ell < k\}$ 
such that $q^{**}_{\alpha_{i,k}} \in G$ and proceed to the next stage of the 
recursive construction. At the end of the construction, our recursion hypothesis 
ensures that, for all $a \in \prod_{i < d} A_i$, we have $q_a \in G$ and hence 
$c \restriction \prod_{i < d} A_i$ is constant, taking value $j^*$. It follows 
that $c_0 \restriction \prod_{i < d} Y_i$ is constant, also taking value $j^*$. 
Finally, our construction ensures that for every $i < d$ and every $t \in 
{^{<\omega}} \omega$, there is $x \in Y_i$ extending $s_i{}^\frown t$, and hence 
$Y_i$ is $s_i$-dense in $[T_i]$, as desired. This completes the proof of Theorem A(3).

\begin{corollary}
	For every $1 \leq d < \omega$, it is consistent that $2^{\aleph_0} = 
	\aleph_{d}$ and, for every $d$-sequence $\vec{X}$ of perfect Polish spaces, 
	$\mathsf{DDF}(\vec{X})$ is weakly $\omega$-partition regular.
\end{corollary}

\section{$\mathsf{PG}_d$ and coding trees}
\label{Sec:coding}

We conclude with a brief discussion of using $\mathsf{PG}$ to prove more detailed versions of the Halpern-L\"auchli theorem pertaining to \emph{coding trees}. In recent years, coding trees in various forms have been developed in \cite{dobrinen} and \cite{zucker} to code countable structures in a finite binary language, such as the Rado graph or Henson's triangle-free graph. Halpern-L\"auchli theorems for the \emph{strict similarity types} of \cite{dobrinen} or the \emph{aged embeddings} of \cite{zucker}, proven using a Harrington-style forcing argument, form the pigeon-hole principle used to show that certain Fra\"iss\'e classes have finite big Ramsey degrees. Indeed, this is a major motivation for trying to develop new proofs of the Halpern-L\"auchli theorem, as these new proofs might generalize to previously unknown settings.

With this in mind, we show how $\mathsf{PG}$ can be used to prove a version of $\mathsf{HL}$ for coding trees of a simple form. Fix $0< d, k< \omega$, and for each $i< d$, let $T_i$ be a copy of the tree ${}^{{<}\omega}k$. We assume that $T_i\cap T_j = \emptyset$ for $i\neq j< d$. We now enrich each $T_i$ to a structure $\mathbf{T}_i$ by declaring that for each $m< \omega$, at most one node of $\bigcup_{i< d} T_i(m)$ is a \emph{coding node}. Write $\mathbf{T} = \langle \mathbf{T}_0,...,\mathbf{T}_{d-1}\rangle$ for this sequence of structures. If $\bigcup_{i< d} T_i(m)$ contains a coding node, we write $c_{\mathbf{T}}(m)$ for this node. We assume that for every $i< d$ and every $t\in T_i$, there is some $n< \omega$ so that $c_{\mathbf{T}}(n)\in T_i$ and $t\sqsubseteq c_{\mathbf{T}}(n)$. 

An \emph{embedding} of $\mathbf{T}$ into itself is an injection $f\colon \bigcup_{i< d} T_i\to \bigcup_{i< d} T_i$ satisfying the following properties:
\begin{enumerate}
	\item 
	$f[T_i]\subseteq T_i$ for each $i< d$,
	\item
	$f$ preserves tree order, meets, relative levels, and lexicographic order. Write $\tilde{f}\colon \omega\to \omega$ for the induced function on levels.
	\item
	If $c_{\mathbf{T}}(m)$ exists, then so does $c_{\mathbf{T}}(\tilde{f}(m))$, and we have $f(c_{\mathbf{T}}(m)) = c_{\mathbf{T}}(\tilde{f}(m))$.
\end{enumerate}
Write $\mathrm{Emb}(\mathbf{T}, \mathbf{T})$ for the set of embeddings of $\mathbf{T}$ into itself.

Note that if we remove the extra coding node structure and item $(3)$ from the above, then items $(1)$ and $(2)$ describe strong subtrees of the form appearing in the ordinary Halpern-L\"auchli theorem. To state the version for these coding trees, first observe that the level product $\bigotimes \mathbf{T}_i$ now contains $d+1$ different types of elements; given $(t_0,...,t_{d-1})\in \bigotimes \mathbf{T}_i$, either none of the $t_i$ is a coding node, or exactly one of the $t_i$ is a coding node. We refer to the former case as \emph{type $-1$} and the latter as \emph{type $i$} for a given $i< d$. Given $\mathsf{p}\in \{-1,...,d-1\}$, let $\mathbf{T}(\mathsf{p})\subseteq \bigotimes \mathbf{T}_i$ denote those tuples of type $\mathsf{p}$. 

\begin{theorem}
	\label{Thm:CodingTreeHL}
	Given $\mathsf{p}\in \{-1, 0,...,d-1\}$, $r< \omega$, and a coloring $\gamma\colon \mathbf{T}(\mathsf{p})\to r$, there is $f\in \mathrm{Emb}(\mathbf{T}, \mathbf{T})$ such that $f[\mathbf{T}(\mathsf{p})]\subseteq \mathbf{T}(\mathsf{p})$ is monochromatic for $\gamma$. 
\end{theorem}

We opt to give a relatively straightforward proof using $\mathsf{PG}_d$. Then by Shoenfield absoluteness, this yields a ZFC proof. 

\begin{proof}
	For each $i< d$, let $X_i\subseteq [T_i]$ consist of those branches which contain infinitely many coding nodes. Then $X_i$ is a dense $G_\delta$ subspace of $[T_i]$, so is itself a perfect Polish space. In the case where $\mathsf{p} = -1$, the proof is almost identical to that of Proposition~\ref{Prop:PGHL}, the key difference being that one works inside the space $\prod_{i< d} X_i$ so that when building $f$, we can ensure there are coding nodes where we need them.
	
	So now assume that without loss of generality $\mathsf{p} = 0$. For each $y\in X_0$, let $U\colon X_0\to \beta \omega$ be a function such that for every $y\in X_0$, we have that $C(y):= \{n< \omega: c_{\mathbf{T}}(n)\in y\}\in U(y)$. Form a coloring $\tilde{\gamma}\colon \prod_{i< d} X_i$ by setting $\tilde{\gamma}(y_0,...,y_{d-1}) = j< r$ iff $\{n\in C(y_0): \gamma(y_0(n),...,y_{d-1}(n)) = j\}\in U(y_0)$. Using $\mathsf{PG}_d$, find somewhere dense sets $Y_i\subseteq X_i$ for each $i< d$ such that $\prod_{i<d} Y_i$ is monochromatic, say with color $j< r$. Say that $(s_0,...,s_{d-1})\in \bigotimes \mathbf{T}_i$ is such that each $Y_i$ is dense above $s_i$. We now proceed to define $f\in \mathrm{Emb}(\mathbf{T}, \mathbf{T})$. Suppose $m< \omega$ and that $f$ has been defined on $\bigcup_{\ell< m} \mathbf{T}_i(\ell)$. If $m = 0$, then letting $\emptyset_i\in \mathbf{T}_i(0)$ denote the root, define $f'(\emptyset_i) = s_i$. If $m> 0$, we define the map $f'\colon \bigcup_{i< d}\mathbf{T}_i(m)$, where given $t\in \bigcup_{i<d}\mathbf{T}_i(m)$ with $t = s^\frown b$ for some $s\in \bigcup_{i<d} \mathbf{T}_i(m-1)$ and some $b< k$, then we set $f'(t) = f(s)^\frown b$. If $\mathbf{T}_0(m)$ does not contain a coding node, then we can define $f$ on $\bigcup_{i< d} \mathbf{T}_i(m)$ as follows. If $c_\mathbf{T}(m)$ exists, pick any $n< \omega$ such that $c_\mathbf{T}(n)\sqsupseteq f'(c_\mathbf{T}(m))$, set $f(c_\mathbf{T}(m)) = c_\mathbf{T}(n)$, and for every other $t\in \bigcup_{i< d}\mathbf{T}_i(m)$, let $f(t)\in \suc(f'(t), n)$ be any node. If $c_\mathbf{T}(m)$ does not exist, then pick any large enough $n< \omega$ and simply let $f(t)\in \suc(f'(t), n)$. If $c_\mathbf{T}(m)\in \mathbf{T}_0(m)$, first for each $i< d$ and each $t\in \mathbf{T}_i(m)$, fix a branch $y_t\in Y_i$ with $f'(t)\in y_0$. Writing $y = y_{c_\mathbf{T}(m)}$, then for any choice of $t_i\in \mathbf{T}_i(m)$ for $0< i< d$, we have $\{n\in C(y): \gamma(y(n), y_{t_i}(n),...,y_{t_{d-1}}(n)) = j\}\in U(y)$. Hence we can find $n\in C(y)$ so that for every $0< i< d$ and every $t_i\in \mathbf{T}_i(m)$, we have that $\gamma(y(n), y_{t_1}(n),...,y_{t_{d-1}}(n)) = j$. We then set $f(t) = y_t(n)$ for every $t\in \bigcup_{i<d} \mathbf{T}_i(m)$. 
\end{proof} 

One major difficulty in generalizing the above argument to the more general coding trees and aged embeddings of \cite{zucker} is that in general, the forcing one needs to use is not Cohen forcing. This suggests that rather than the principle $\mathsf{PG}$, one would ask for a combinatorial principle corresponding to each specific type of forcing used. Nonetheless, the following seems like a worthwhile question to ask.

\begin{question}
	\label{Que:OtherPrinciples}
	Is there a family of consistent combinatorial principles about partitions of structures on products of Polish spaces which implies all of the variants of the Halpern-L\"auchli theorem appearing in \cite{dobrinen} and \cite{zucker}?
\end{question}

\bibliographystyle{amsplain}
\bibliography{bib}

\providecommand{\bysame}{\leavevmode\hbox to3em{\hrulefill}\thinspace}
\providecommand{\MR}{\relax\ifhmode\unskip\space\fi MR }
\providecommand{\MRhref}[2]{%
  \href{http://www.ams.org/mathscinet-getitem?mr=#1}{#2}
}
\providecommand{\href}[2]{#2}
\begin{thebibliography}{10}

\bibitem{bbmt}
Nathaniel Bannister, Jeffrey Bergfalk, Justin~Tatch Moore, and Stevo
  Todorcevic, \emph{A descriptive approach to higher derived limits},  (2022),
  Preprint.

\bibitem{svhdlwlc}
Jeffrey Bergfalk, Michael Hru\v{s}\'{a}k, and Chris Lambie-Hanson,
  \emph{Simultaneously vanishing higher derived limits without large
  cardinals}, J. Math. Log. (2022), To appear.

\bibitem{DobrinenSurvey}
Natasha Dobrinen, \emph{Forcing in {R}amsey theory}, RIMS Kokyuroku
  \textbf{2042} (2017), 17--33.

\bibitem{dobrinen}
\bysame, \emph{The {R}amsey theory of the universal homogeneous triangle-free
  graph}, J. Math. Log. \textbf{20} (2020), no.~2, 2050012.

\bibitem{halpern_lauchli}
J.~D. Halpern and H.~L\"{a}uchli, \emph{A partition theorem}, Trans. Amer.
  Math. Soc. \textbf{124} (1966), 360--367. \MR{200172}

\bibitem{halpern_levy}
J.~D. Halpern and A.~L\'{e}vy, \emph{The {B}oolean prime ideal theorem does not
  imply the axiom of choice}, Axiomatic {S}et {T}heory ({P}roc. {S}ympos.
  {P}ure {M}ath., {V}ol. {XIII}, {P}art {I}, {U}niv. {C}alifornia, {L}os
  {A}ngeles, {C}alif., 1967), Amer. Math. Soc., Providence, R.I., 1971,
  pp.~83--134. \MR{0284328}

\bibitem{hindman_leader_strauss}
Neil Hindman, Imre Leader, and Dona Strauss, \emph{Pairwise sums in colourings
  of the reals}, Abh. Math. Semin. Univ. Hambg. \textbf{87} (2017), no.~2,
  275--287. \MR{3696151}

\bibitem{komjath}
P\'{e}ter Komj\'{a}th, \emph{Three clouds may cover the plane}, vol. 109, 2001,
  Dedicated to Petr Vop\v{e}nka, pp.~71--75. \MR{1835239}

\bibitem{inf_mono_sumsets}
P\'{e}ter Komj\'{a}th, Imre Leader, Paul~A. Russell, Saharon Shelah,
  D\'{a}niel~T. Soukup, and Zolt\'{a}n Vidny\'{a}nszky, \emph{Infinite
  monochromatic sumsets for colourings of the reals}, Proc. Amer. Math. Soc.
  \textbf{147} (2019), no.~6, 2673--2684. \MR{3951442}

\bibitem{higher_delta_systems}
Chris Lambie-Hanson, \emph{Higher-dimensional {D}elta-systems}, Order (2022),
  To appear.

\bibitem{milliken}
Keith~R. Milliken, \emph{A {R}amsey theorem for trees}, J. Combin. Theory Ser.
  A \textbf{26} (1979), no.~3, 215--237. \MR{535155}

\bibitem{raghavan_todorcevic_2020}
Dilip Raghavan and Stevo Todor\v{c}evi\'c, \emph{Proof of a conjecture of
  galvin}, Forum of Mathematics, Pi \textbf{8} (2020), e15.

\bibitem{raghavan_todorcevic_hd}
\bysame, \emph{Galvin's problem in higher dimensions},  (2022), Preprint.

\bibitem{stefanovic}
Nedeljko Stefanovi\'{c}, \emph{Alternatives to {H}alpern and {L}\"{a}uchli's
  theorem},  (2022), Preprint.

\bibitem{stevo_book}
Stevo Todorcevic, \emph{Introduction to {R}amsey spaces}, Annals of Mathematics
  Studies, vol. 174, Princeton University Press, Princeton, NJ, 2010.
  \MR{2603812}

\bibitem{todorcevic_farah}
S.~Todor\v{c}evi\'{c} and I.~Farah, \emph{Some applications of the method of
  forcing}, Yenisei Series in Pure and Applied Mathematics, Yenisei, Moscow;
  Lyc\'{e}e, Troitsk, 1995. \MR{1486583}

\bibitem{zhang}
Jing Zhang, \emph{Monochromatic sumset without large cardinals}, Fund. Math.
  \textbf{250} (2020), no.~3, 243--252. \MR{4107536}

\bibitem{hl2d}
Andy Zucker, \emph{A new proof of the 2-dimensional {H}alpern-{L}\"{a}uchli
  theorem},  (2017), Unpublished note, available at
  \url{https://www.math.cmu.edu/~andrewz/HL2d.pdf}.

\bibitem{zucker}
\bysame, \emph{On big {R}amsey degrees for binary free amalgamation classes},
  Adv. Math. \textbf{408A} (2022), 108585.

\end{thebibliography}

\end{document}